\documentclass[12pt]{amsart}
\usepackage{amssymb}
\theoremstyle{plain}
 \newtheorem{thm}{Theorem}[section]
 
 \newtheorem{lem}[thm]{Lemma}
 \newtheorem{prop}[thm]{Proposition}

\theoremstyle{definition}
 \newtheorem{ex}{Example}[section]
\theoremstyle{remark}

\begin{document}
\title[a note on the Jensen inequality for 
self-adjoint operators. II.]
{a note on the Jensen inequality for self-adjoint operators. II.}
\author[
Tomohiro Hayashi]{{Tomohiro Hayashi} }
\address[Tomohiro Hayashi]
{Nagoya Institute of Technology, 
Gokiso-cho, Showa-ku, Nagoya, Aichi, 466-8555, Japan}
\email[Tomohiro Hayashi]{hayashi.tomohiro@nitech.ac.jp}

\baselineskip=17pt
\maketitle

\begin{abstract}
This is a continuation of our previous paper. 
We consider a certain order-like relation 
for positive operators on a Hilbert space. This 
relation is defined by using the Jensen inequality 
with respect to the square-root function. 
We show that  
this relation is antisymmetric if the operators are 
invertible. 
\end{abstract}

\section{Introduction}
This is a continuation of our previous paper \cite{H2}. 
Let $f(t)$ be a continuous, increasing concave function 
on the half line $[0,\infty)$ and let $A$ 
and $B$ 
be 
bounded 
self-adjoint operators on a Hilbert space 
${\frak H}$ with an inner product 
$\langle\cdot,\cdot\rangle$. 
In the previous paper, we consider the following problem. 
If $A$ and $B$ 
satisfy 
$
\langle 
f(A)\xi,\xi
\rangle
\leq
f(\langle 
B\xi,\xi
\rangle)
$ and 
$
\langle 
f(B)\xi,\xi
\rangle
\leq
f(\langle 
A\xi,\xi
\rangle)
$ 
for any unit vector $\xi\in {\frak H}$, 
can we conclude $A=B$? 
This problem was suggested by Professor Bourin~\cite{B}. 
In \cite{H2} we solved this problem affirmatively in the 
finite-dimensional case. We also dealt with 
some related problem in the 
infinite-dimensional case, but we could not get a 
complete answer. 
In this paper we consider the case $f(t)=\sqrt{t}$ 
and we solve this problem affirmatively under the assumption 
that two positive operators $A$ and $B$ are both invertible. 

For two positive operators $A$ and $B$, we 
introduce the new relation $A\trianglelefteq B$ 
defined by 
$
\langle A^{\frac{1}{2}}\xi,\xi\rangle
\leq 
\langle B\xi,\xi\rangle^{\frac{1}{2}}
$ 
for any unit vector $\xi\in {\frak H}$. 
Using this notation, 
we can restate the above problem as follows. 
If $A$ and $B$ 
satisfy 
$
A\trianglelefteq B
$ and 
$
B\trianglelefteq A
$, 
can we conclude $A=B$? We will show that this is true 
when $A$ and $B$ are both invertible. 
Here we remark that the usual order 
$A\leq B$ implies $A\trianglelefteq B$ thanks to 
the Jensen inequality. 
However the relation $\trianglelefteq $ is not an 
order relation. Indeed 
we will construct 
positive 
matrices $A$, $B$ and $C$ 
such that 
both $A\trianglelefteq B$ and 
$B\trianglelefteq C$ hold while 
$A\trianglelefteq C$ does not 
hold. 

The author would like to express his hearty gratitude to Professor 
Tsuyoshi Ando for valuable comments.

\section{Main Result}

Throughout this paper we assume that 
the readers are familiar with 
basic notations and results on operator 
theory. We refer the readers to 
Conway's book~\cite{C}.

We denote 
by ${\frak H}$ a 
(finite or infinite dimensional) 
complex Hilbert space 
and by $B({\frak H})$ 
all bounded linear operators on it. 
The operator norm of 
$A\in B({\frak H})$ 
is denoted by 
$||A||$. 
The inner product and the norm 
for two vectors $\xi,\eta\in {\frak H}$
are denoted by 
$\langle \xi,\eta\rangle$ and $||\xi||$ 
respectively. 
We denote the defining function 
for an interval $[a,b)$ 
by 
$\chi_{[a,b)}(t)$. 
We define the absolute value for a bounded linear 
operator $X$ by 
$
|X|=(X^{*}X)^{\frac{1}{2}}
$.

If two positive operators $A,B\in B({\frak H})$ satisfy 
$$
\langle A^{\frac{1}{2}}\xi,\xi\rangle
\leq 
\langle B\xi,\xi\rangle^{\frac{1}{2}}
$$ 
for any unit vector $\xi\in {\frak H}$, we write 
$$A\trianglelefteq B.$$
The usual order $A\leq B$ implies that $A\trianglelefteq B$. 
This is a consequence of the famous Jensen inequality 
as follows. 
$$
\langle A^{\frac{1}{2}}\xi,\xi\rangle
\leq 
\langle A\xi,\xi\rangle^{\frac{1}{2}}
\leq 
\langle B\xi,\xi\rangle^{\frac{1}{2}}.
$$ 
Here we remark that the relation 
$
\trianglelefteq
$ 
is not an order relation. Indeed there exit positive matrices 
$A$, $B$ and $C$ such that both $A\trianglelefteq B$ and 
$B\trianglelefteq C$ hold while $A\trianglelefteq C$ does not 
hold. See Example 2.1. 

The following is the main result of this paper. 
\begin{thm}
Let $A,B\in B({\frak H})$ be two positive operators such 
that $A$ is invertible. If they satisfy 
$A\trianglelefteq B$ and $B\trianglelefteq A$, 
then we have $A=B$. 
\end{thm}

Here we remark that 
it is hard to remove the assumption of invertibility. 
See Example 2.1.

\begin{prop}[Ando~\cite{A2}]
For two positive operators $A,B\in B({\frak H})$, 
the following conditions are equivalent. 
\begin{enumerate}
\item $A^{2}\trianglelefteq B^{2}$. 
\item $A\leq \frac{1}{2t}B^{2}
+\frac{t}{2}$ for any positive number $t$. 
\item There exists a contraction $C$ satisfying 
$
CB+BC^{*}=2A
$ 

\end{enumerate}
\end{prop}

\begin{proof}
The equivalence (i)$\Leftrightarrow$(ii) is shown in 
\cite{A1}. (See also \cite{H2} Lemma 3.2.) 

Suppose that there exists a contraction $C$ satisfying 
$
CB+BC^{*}=2A
$. Since 
$$
0\leq(CB-t)^{*}(CB-t)=BC^{*}CB+t^{2}-t(CB+BC^{*}),
$$ we see that 
$$
2tA=t(CB+BC^{*})\leq BC^{*}CB+t^{2}\leq B^{2}+t^{2}.
$$ 
Therefore the implication (iii)$\Rightarrow$(ii) holds. 

Finally we will show (ii)$\Rightarrow$(iii). We remark that 
the inequality 
$$
B^{2}+t^{2}-2tA\geq 0
$$ 
holds 
for any real number $t$. Thus 
by the 
operator-valued Fejer-Riesz theorem 
(\cite{RR}Theorem 3.3) 
there exist 
two bounded linear operators $X$ and $Y$ such that 
$$
B^{2}+t^{2}-2tA=(X-tY)^{*}(X-tY)
=X^{*}X+t^{2}Y^{*}Y-t(X^{*}Y+Y^{*}X).
$$ 
Therefore we have 
$B=|X|$, $|Y|=1$ and $2A=X^{*}Y+Y^{*}X$. 
Here we remark that $Y$ is a contraction because 
$|Y|=1$. 
Take a polar decomposition 
$
X=U|X|=UB
$ 
where $U$ is a partial isometry. Then we get 
$$
2A=B(U^{*}Y)+(Y^{*}U)B.
$$ 
Since $U^{*}Y$ is a contraction, we are done. 

\end{proof}

\begin{lem} 
Let $c$ and $\epsilon$ be positive numbers 
such that $\epsilon<c$. Then 
$$
2t\lambda-t^{2}>0
\ \ {\text{and}}\ \ 
\frac{\lambda^{2}}{2t}+\frac{t}{2}-
(2t\lambda-t^{2})^{\frac{1}{2}}\geq 0
$$ 
for any $c+\epsilon\leq t,\lambda\leq 2c$. 
Further there exists a positive number 
$d$ satisfying 
$$
\frac{\lambda^{2}}{2t}+\frac{t}{2}-
(2t\lambda-t^{2})^{\frac{1}{2}}\leq 
\frac{d}{2}(t-\lambda)^{2}\eqno{(9)}
$$ 
for any $c+\epsilon\leq t,\lambda\leq 2c$.
\end{lem}

\begin{proof}
The proof is same as that of \cite{H2} 
Lemma 3.4. 

Since $c+\epsilon\leq t,\lambda\leq 2c$, we have 
$$
2t\lambda-t^{2}=t(2\lambda-t)\geq 
(c+\epsilon)\{2(c+\epsilon)-2c\}=
2(c+\epsilon)\epsilon>0.
$$ 
Next by the arithmetic-geometric mean inequality 
we have 
$
\frac{\lambda^{2}}{2t}+\frac{t}{2}\geq \lambda
$ 
and obviously 
$
\lambda^{2}\geq 2t\lambda-t^{2}
$, 
so that 
$
\lambda\geq (2t\lambda-t^{2})^{\frac{1}{2}}
$. 

Now we set 
$$
k(t,\lambda)=
\frac{d}{2}(t-\lambda)^{2}-
\frac{\lambda^{2}}{2t}-\frac{t}{2}+
(2t\lambda-t^{2})^{\frac{1}{2}}.
$$
Then we compute 
$$
\frac{\partial}{\partial t}k(t,\lambda)
=
d(t-\lambda)+
\frac{\lambda^{2}}{2t^{2}}-
\frac{1}{2}+
\frac{\lambda-t}{(2t\lambda-t^{2})^{\frac{1}{2}}}
$$ 
and 
$$
\frac{\partial^{2}}{\partial t^{2}}k(t,\lambda)
=
d-\frac{\lambda^{2}}{t^{3}}
+\frac{-(2t\lambda-t^{2})^{\frac{1}{2}}
-(\lambda-t)^{2}(2t\lambda-t^{2})^{-\frac{1}{2}}
}{2t\lambda-t^{2}}.
$$
Since $c+\epsilon\leq t\leq 2c$ and 
$c+\epsilon\leq \lambda\leq 2c$, 
we see that 
$
2t\lambda-t^{2}=t(2\lambda-t)\geq 
(c+\epsilon)\{2(c+\epsilon)-2c\}=
2(c+\epsilon)\epsilon>0
$. Thus the two-variable function 
$$
-\frac{\lambda^{2}}{t^{3}}
+\frac{-(2t\lambda-t^{2})^{\frac{1}{2}}
-(\lambda-t)^{2}(2t\lambda-t^{2})^{-\frac{1}{2}}
}{2t\lambda-t^{2}}
$$ 
is bounded below on the intervals $c+\epsilon\leq t\leq 2c$ and 
$c+\epsilon\leq \lambda\leq 2c$. 
Therefore we can find a positive constant $d$ such that 
$\frac{\partial^{2}}{\partial t^{2}}k(t,\lambda)>0$ 
on the intervals $c+\epsilon\leq t\leq 2c$ and 
$c+\epsilon\leq \lambda\leq 2c$. Then 
$k(t,\lambda)$ 
is convex with respect to $t$. 
Since 
$\frac{\partial}{\partial t}k(t,\lambda)|_{t=\lambda}=0$, 
$k(t,\lambda)$ 
in $t$ is decreasing for $c+\epsilon\leq t\leq\lambda$ 
and increasing 
for $\lambda\leq t\leq c$ so that 
$
k(t,\lambda)\geq k(\lambda,\lambda)=0
$. Thus we are done. 
\end{proof}

\begin{lem}
Let $A,B\in B({\frak H})$ be positive invertible operators 
such that $c+\epsilon\leq A\leq 2c$ for some positive numbers 
$\epsilon<c$. If they satisfy 
$$
(2tA-t^{2})^{\frac{1}{2}}
\leq B 
\leq 
\frac{A^{2}}{2t}+\frac{t}{2}
$$ 
for any positive number $t$ on the 
interval $c+\epsilon\leq t\leq 2c$, 
then we have $A=B$. 
\end{lem}

\begin{proof}
The proof is essentially same as that of 
\cite{A1,H1,H2}. 

First we will show 
that there exists a positive constant $d$ 
satisfying 
$$
||PBP-(PB^{-1}P)^{-1}||\leq d||tP-AP||^{2}\eqno{(1)}
$$ 
for any $c+\epsilon\leq t\leq 2c$ and any 
spectral projection $P$ of $A$, 
where we use $(PB^{-1}P)^{-1}$ to denote the 
inverse of $PB^{-1}P$ on $P{\frak H}$. 
In the following we use commutativity of 
$A$ and $P$ without any particular mention.

By assumption we have two inequalities 
$$
(2tA-t^{2})^{\frac{1}{2}}
\leq B 
\leq 
\frac{A^{2}}{2t}+\frac{t}{2}\eqno{(2)}
$$ 
and 
$$
2t(A^{2}+t^{2})^{-1}
\leq B^{-1} 
\leq 
(2tA-t^{2})^{-\frac{1}{2}}.\eqno{(3)}
$$ 
Here we remark that $(2tA-t^{2})^{-\frac{1}{2}}$ is a 
bounded operator because 
$2tA-t^{2}=t(2A-t)$ and $2A\geq 2(c+\epsilon)
>2c
\geq t>0$. On the other hand we have 
$$
(2tA-t^{2})^{\frac{1}{2}}
\leq A 
\leq 
\frac{A^{2}}{2t}+\frac{t}{2}.\eqno{(4)}
$$ 
By the inequalities (2) and (4), we see that 
$$
\pm(AP-PBP)\leq 
\frac{(AP)^{2}}{2t}+\frac{t}{2}P-
(2tAP-t^{2}P)^{\frac{1}{2}}
$$ 
and hence 
$$
||AP-PBP||\leq 
\Bigl|\Bigl|\frac{(AP)^{2}}{2t}+\frac{t}{2}P-
(2tAP-t^{2}P)^{\frac{1}{2}}\Bigr|\Bigr|.\eqno{(5)}
$$ 
By the inequality (3) we have 
$$
2t(A^{2}+t^{2})^{-1}P
\leq PB^{-1}P 
\leq 
(2tA-t^{2})^{-\frac{1}{2}}P
$$ 
and hence 
$$
(2tAP-t^{2}P)^{\frac{1}{2}}
\leq (PB^{-1}P)^{-1}
\leq 
\frac{(AP)^{2}}{2t}+\frac{t}{2}P.\eqno{(6)}
$$
By the inequalities (4) and (6) we have 
$$
\pm(AP-(PB^{-1}P)^{-1})\leq \frac{(AP)^{2}}{2t}+\frac{t}{2}P-
(2tAP-t^{2}P)^{\frac{1}{2}}
$$ 
and hence 
$$
||AP-(PB^{-1}P)^{-1}||\leq 
\Bigl|\Bigl|\frac{(AP)^{2}}{2t}+\frac{t}{2}P-
(2tAP-t^{2}P)^{\frac{1}{2}}\Bigr|\Bigr|.\eqno{(7)}
$$ 
By the inequalities (5) and (7) we get 
$$
||PBP-(PB^{-1}P)^{-1}||\leq 
2\Bigl|\Bigl|\frac{(AP)^{2}}{2t}+\frac{t}{2}P-
(2tAP-t^{2}P)^{\frac{1}{2}}\Bigr|\Bigr|.\eqno{(8)}
$$ 
By the inequality (8) and Lemma 2.3 we have 
shown the inequality (1). 

By the well-known formula known as 
Schur multiplier, we have 
$$
(PB^{-1}P)^{-1}
=PBP
-PBP^{\perp}
(P^{\perp}BP^{\perp})^{-1}
P^{\perp}BP
$$ 
and hence 
$$
PBP-(PB^{-1}P)^{-1}
=PBP^{\perp}
(P^{\perp}BP^{\perp})^{-1}
P^{\perp}BP\eqno{(9)}
$$ 
with $P^{\perp}=1-P$. Therefore 
by inequality (1) and (9) we see that 
$$
||PBP^{\perp}(P^{\perp}BP^{\perp})^{-1}
P^{\perp}BP||\leq d||tP-AP||^{2}\eqno{(10)}
$$
Then by the inequality (10) we compute 
\begin{align*}
||P^{\perp}BP||^{2}
&=||(P^{\perp}BP^{\perp})^{1/2}
(P^{\perp}BP^{\perp})^{-1/2}
P^{\perp}BP||^{2}\\
&\leq 
||B||\cdot
||(P^{\perp}BP^{\perp})^{-1/2}
P^{\perp}BP||^{2}\\
&=||B||\cdot
||PBP^{\perp}
(P^{\perp}BP^{\perp})^{-1}
P^{\perp}BP||\\
&\leq 
d||B||\cdot||tP-AP||^{2} 
\end{align*} 
and hence 
$$
||P^{\perp}BP||^{2}\leq 
d||B||\cdot||tP-AP||^{2}.\eqno{(11)}
$$
For each integer $n$, let $P_{i}\ \ (i=1,2,\cdots,n)$ 
be the spectral projections of $A$ corresponding to 
the interval 
$[c+\epsilon+\frac{(i-1)\{2c-(c+\epsilon)\}}{n},
c+\epsilon+\frac{i\{2c-(c+\epsilon)\}}{n}]$. 
Then we have $\sum_{i}P_{i}=1$ and 
$$
||
t_{i}P_{i}-AP_{i}
||
\leq 
\frac{c-\epsilon}{n}\eqno{(12)}
$$
where $t_{i}=
c+\epsilon+\frac{(i-1)\{2c-(c+\epsilon)\}}{n}
$. 
By the inequalities (11) and (12) 
we see that 
\begin{align*}
||\sum_{i=1}^{n}
P_{i}^{\perp}BP_{i}||^{2}
&=||\{\sum_{i=1}^{n}
P_{i}^{\perp}BP_{i}\}\{
\sum_{j=1}^{n}P_{j}BP_{j}^{\perp}\}||\\
&=
||\sum_{i=1}^{n}
P_{i}^{\perp}BP_{i}BP_{i}^{\perp}||\\
&\leq 
\sum_{i=1}^{n}
||P_{i}^{\perp}BP_{i}BP_{i}^{\perp}||\\
&=\sum_{i=1}^{n}
||P_{i}^{\perp}BP_{i}||^{2}\\
&\leq 
\sum_{i=1}^{n}d||B||\cdot||t_{i}P_{i}-AP_{i}||^{2}\\
&\leq
\sum_{i=1}^{n}d||B||\cdot
\frac{(c-\epsilon)^{2}}{n^{2}}
=
d||B||\cdot
\frac{(c-\epsilon)^{2}}{n}
\end{align*}
and hence  
$$
||\sum_{i=1}^{n}
P_{i}^{\perp}BP_{i}||^{2}\leq 
d||B||\cdot
\frac{(c-\epsilon)^{2}}{n}
.\eqno{(13)}
$$
Since 
$$
A-B=\sum_{i=1}^{n}(AP_{i}-
P_{i}BP_{i})+\sum_{i=1}^{n}
P_{i}^{\perp}BP_{i},
$$
by (13) we see that 
\begin{align*}
||A-B||&\leq||\sum_{i=1}^{n}(AP_{i}-
P_{i}BP_{i})||+||\sum_{i=1}^{n}
P_{i}^{\perp}BP_{i}||\\
&\leq 
\sup_{i}||AP_{i}-
P_{i}BP_{i}||+
\Bigl(
d||B||\cdot
\frac{(c-\epsilon)^{2}}{n}
\Bigr)^{\frac{1}{2}}
\end{align*}
On the other hand by (5) and Lemma 2.3 
we have 
$$
||AP_{i}-
P_{i}BP_{i}||\leq \frac{d}{2}
||tP_{i}-AP_{i}||^{2}\leq 
\frac{d}{2}
\Bigl(
\frac{c-\epsilon}{n}
\Bigr)^{2}
$$
Thus we get 
$$
||A-B||\leq 
\frac{d}{2}
\Bigl(
\frac{c-\epsilon}{n}
\Bigr)^{2}
+\Bigl(
d||B||\cdot
\frac{(c-\epsilon)^{2}}{n}
\Bigr)^{\frac{1}{2}}
$$ 
By tending $n\rightarrow\infty$ we see that 
$A=B$. 
\end{proof}

\begin{lem}
Let $A,B\in B({\frak H})$ be positive operators 
satisfying $A\trianglelefteq B$. If $A$ is invertible, 
then $B$ is also invertible. 
\end{lem}

\begin{proof}
By assumption, there exists a positive number $c$ which 
satisfies $c\leq A$. Then we have 
$$
c^{\frac{1}{2}}\langle \xi,\xi\rangle\leq 
\langle A^{\frac{1}{2}}\xi,\xi\rangle
\leq 
\langle B\xi,\xi\rangle^{\frac{1}{2}}
$$ 
for any unit vector $\xi\in {\frak H}$. 
Therefore $B$ is invertible. 
\end{proof}

\begin{lem}
Let $A$ be a positive operator and let $C$ be a contraction. 
If they satisfy 
$
CA+AC^{*}=2A
$, 
then we have $CP=P$ where $P$ is the range projection of $A$. 
\end{lem}

\begin{proof}
This is a kind of triangle equality. 
The proof is implicitly contained in \cite{AH}. By assumption 
we have 
$
(C-1)A=A(1-C^{*})
$. This means that the operator 
$(C-1)A$ is skew-selfadjoint. Therefore the spectrum 
$\sigma((C-1)A)$ is contained in $i{\Bbb R}$. On the other hand 
we see that 
$
\sigma((C-1)A)\cup\{0\}=
\sigma(A^{\frac{1}{2}}(C-1)A^{\frac{1}{2}})\cup\{0\}
$, and by \cite{AH} Lemma 2.2 we have 
$
\sigma(A^{\frac{1}{2}}(C-1)A^{\frac{1}{2}})\cap i{\Bbb R}
=\{0\}.
$ Therefore we conclude that 
$\sigma((C-1)A)=\{0\}$. Since $(C-1)A$ is skew-selfadjoint, 
we see that $(C-1)A=0$. 
\end{proof}

\bigskip 

\begin{proof}[Proof of Theorem 2.1]
By Lemma 2.5 we may assume that both $A$ and $B$ 
are invertible. 
It is enough to show that 
two relations 
$A^{2}\trianglelefteq B^{2}$ and 
$B^{2}\trianglelefteq A^{2}$ ensure that 
$A=B$ for positive invertible operators 
$A$ and $B$. 

By Proposition 2.2 we have two inequalities 
$$
A 
\leq 
\frac{B^{2}}{2t}+\frac{t}{2}\eqno{(14)}
$$ 
and 
$$
B 
\leq 
\frac{A^{2}}{2t}+\frac{t}{2}\eqno{(15)}
$$
for any positive number $t$. 
Since $A$ is positive invertible, there exists a positive 
number $c$ satisfying 
$A\geq c$. Let $\epsilon$ be a positive number 
with $\epsilon<c$. 
It follows from (14) and Lemma 2.3 
$$
0\leq 
2tA-t^{2}\leq B^{2}
$$ 
for any $c+\epsilon\leq t\leq 2c$. 
Then since the map 
$
X\longmapsto X^{\frac{1}{2}}
$ 
is order-preserving in the cone of positive 
operators, 
we have from (15)
$$
(2tA-t^{2})^{\frac{1}{2}}
\leq B 
\leq 
\frac{A^{2}}{2t}+\frac{t}{2}.
$$ 
for any $c+\epsilon\leq t\leq 2c$. 
Let $P=\chi_{[c+\epsilon,2c]}(A)$. Then we have 
$$
(2tAP-t^{2}P)^{\frac{1}{2}}
\leq PBP 
\leq 
\frac{(AP)^{2}}{2t}+\frac{t}{2}P
$$ 
and 
$(c+\epsilon)P\leq AP\leq 2cP$. Therefore 
by Lemma 2.4 we have 
$
AP=PBP
$. 
By Proposition 2.2 there exists a contraction $D$ such that 
$$
DA+AD^{*}=2B
$$ 
and hence 
$$
PDPA+APD^{*}P=2PBP=2AP. 
$$
Then by Lemma 2.6 we see that 
$
PDP=P
$. 
Since 
$$
P=PD^{*}PDP\leq PD^{*}DP\leq P,
$$ 
we have 
$
(1-P)DP=0
$ and hence 
$
DP=PDP+(1-P)DP=P
$. By the same argument we see that 
$PD=P$. Therefore we have 
$$
2BP=(DA+AD^{*})P=DPA+AD^{*}P=2AP
$$ 
and hence $BP=PB$. Since $\epsilon$ is arbitrary, 
we have 
$$
A\chi_{(c,2c]}(A)=B\chi_{(c,2c]}(A)=\chi_{(c,2c]}(A)B. 
$$
Since the positive invertible operators 
$A(1-\chi_{(c,2c]}(A))$ and 
$B(1-\chi_{(c,2c]}(A))$ on 
$(1-\chi_{(c,2c]}(A)){\frak H}$ satisfy 
$$
\{A(1-\chi_{(c,2c]}(A))\}^{2}
\trianglelefteq
\{B(1-\chi_{(c,2c]}(A))\}^{2}
$$ 
and 
$$
\{B(1-\chi_{(c,2c]}(A))\}^{2}
\trianglelefteq
\{A(1-\chi_{(c,2c]}(A))\}^{2},
$$ 
by the same argument we see that 
$$
A\chi_{(2c,4c]}(A)=B\chi_{(2c,4c]}(A)=\chi_{(2c,4c]}(A)B.
$$
Therefore by repeating this argument we have $A=B$. 

\end{proof}

\begin{lem}
For any operator $X$, we have 
$$
{\rm Re}X\leq \frac{1}{2t}|X|^{2}
+\frac{t}{2} 
$$ 
for any positive number $t$. 
\end{lem}

\begin{proof}
Since 
$$
0\leq (X-t)^{*}(X-t)=|X|^{2}+t^{2}-2t{\rm Re}X,
$$ 
we are done. 
\end{proof}

\begin{ex}
First we will show that there exist $2\times 2$ positive 
matrices $A$, $B$ and $C$ 
such that 
both $A^{2}\trianglelefteq B^{2}$ and 
$B^{2}\trianglelefteq C^{2}$ hold while 
$A^{2}\trianglelefteq C^{2}$ does not 
hold. 

We set 
$$
X=
\begin{pmatrix}
\sqrt{2}&1\\
0&\sqrt{2}
\end{pmatrix}
,\ \ \ \ \ 
A={\rm Re}X=
\begin{pmatrix}
\sqrt{2}&\frac{1}{2}\\
\frac{1}{2}&\sqrt{2}
\end{pmatrix}
\geq 0
$$ 
and 
$$
B=|X|=
\frac{1}{3}
\begin{pmatrix}
4&\sqrt{2}\\
\sqrt{2}&5
\end{pmatrix}.
$$
By Lemma 2.7 and Proposition 2.2 
we have $A^{2}\trianglelefteq B^{2}$. 
Next we set 
$$
Y=
\frac{1}{3}
\begin{pmatrix}
4&2\sqrt{2}\\
0&5
\end{pmatrix}
$$ 
and $C=|Y|$. Since 
$$
{\rm Re}Y=
\frac{1}{3}
\begin{pmatrix}
4&\sqrt{2}\\
\sqrt{2}&5
\end{pmatrix}
=B,
$$
we have $B^{2}\trianglelefteq C^{2}$. 
Suppose that $A^{2}\trianglelefteq C^{2}$. 
Then by Proposition 2.2 we have 
$$
A\leq \frac{1}{2t}C^{2}
+\frac{t}{2}$$ 
for any positive number $t$. Let 
$
E=
\begin{pmatrix}
1&0\\
0&0
\end{pmatrix}
$. Then we see that 
$
EAE=\sqrt{2}E
$ and 
$
E(\frac{1}{2t}C^{2}
+\frac{t}{2})E=
(\frac{1}{2t}\times\frac{16}{9}
+\frac{t}{2})E
$. Therefore we have 
$$
\sqrt{2}\leq \frac{8}{9t}
+\frac{t}{2}
$$ 
for any positive number $t$. This is impossible 
because the minimal value of the right hand side is 
$\frac{4}{3}$ while $\frac{4}{3}<\sqrt{2}$

Next we show that 
$(A+\epsilon)^{2}\trianglelefteq (B+\epsilon)^{2}$ is not 
valid 
for any positive number $\epsilon$. 
If this is the case, we have 
$$
E(A+\epsilon)E=(\sqrt{2}+\epsilon)E
\leq \frac{1}{2t}E(B+\epsilon)^{2}E
+\frac{t}{2}E
=\Bigl(\frac{9\epsilon^{2}
+24\epsilon+18
}{18t}+\frac{t}{2}\Bigr)E
$$ 
for any positive number $t$. 
Since the minimal value of 
the scalar on 
the right hand side is 
$
\frac{\sqrt{9\epsilon^{2}
+24\epsilon+18}
}{3}
$, we have 
$$
(\sqrt{2}+\epsilon)^{2}
=\epsilon^{2}+2\sqrt{2}\epsilon+2
\leq 
\Bigl(\frac{\sqrt{9\epsilon^{2}
+24\epsilon+18}
}{3}\Bigr)^{2}
=\epsilon^{2}+\frac{8}{3}\epsilon+2.
$$ 
This is obviously wrong because 
$2\sqrt{2}>\frac{8}{3}$. 
This is the reason why we cannot remove the assumption of 
invertibility. In the proof of the main theorem, the 
inequality 
$$
A\leq \frac{1}{2t}B^{2}
+\frac{t}{2}
$$ 
is crucial. So if $A$ is not invertible, we hope that 
the inequality 
$$
A+\epsilon\leq \frac{1}{2t}(B+\epsilon)^{2}
+\frac{t}{2}
$$ holds 
for any small number $\epsilon$. However this is not true 
in general.

\end{ex}

\end{document}